\documentclass{amsart}
\usepackage[T1]{fontenc}
\usepackage{textcomp}
\usepackage[margin=1in]{geometry}
\usepackage{graphicx}
\usepackage{amsmath,amsfonts,amssymb,,amsthm,hyperref,xcolor,array,bigdelim,cleveref,blkarray}
\usepackage{enumitem}
\usepackage[numbers,sort]{natbib}

\definecolor{dartmouthgreen}{rgb}{0.05, 0.5, 0.06}
\definecolor{deepskyblue}{rgb}{0.0, 0.75, 1.0}

\definecolor{aolongcrimson}{rgb}{0.78, 0.12, 0.18}

\usepackage{arydshln} 
\theoremstyle{definition}
\newtheorem{theorem}{Theorem}[section]
\newtheorem{remark}[theorem]{Remark}
\newtheorem{definition}[theorem]{Definition}
\newtheorem{proposition}[theorem]{Proposition}
\newtheorem{example}[theorem]{Example}
\def\KK{\mathbb{K}}
\usepackage{listings}
\definecolor{codedarkgreen}{RGB}{51, 133, 4}
\definecolor{codemaroon}{RGB}{133, 5, 63}
\definecolor{codeteal}{RGB}{0, 145, 109}
\definecolor{codepurple}{RGB}{123, 35, 125}

\lstdefinelanguage{Macaulay2}{
basicstyle=\normalsize\ttfamily,
  alsoletter=",
  classoffset=1,
  keywords={coefficients,toList,matrix,trace,random,norm,sub,apply,setRandomSeed,transpose,det,factor,all,netList,subsets,genericMatrix,needsPackage,presentation,generators,gens,selectInSubring,i1,i2,i3,i4,i5,i6,i7,i8,i9,i10,i11,i12,i13,i14,i15,i16,i17,i18,i19,i20,i21,flatten,ideal,frac,LUdecomposition,diagonalMatrix},
  keywordstyle={\color{blue}},
classoffset=2,
breaklines=true,
morekeywords={"EliminationTemplates","Engine","Brackets"},
keywordstyle={\color{codemaroon}},
classoffset=3,
morekeywords={QQ,CC,ZZ,CacheTable,Matrix,Eliminate},
keywordstyle={\color{codedarkgreen}},
classoffset=4,
morekeywords={restart,false,true,Weights,Limit,Lex,MonomialOrder,CoefficientRing},
keywordstyle={\color{codeteal}},
classoffset=5,
morekeywords={list,for,in,from,to,of},
keywordstyle={\color{codepurple}},
xleftmargin=1em,
xrightmargin=1em,
columns=fullflexible,
keepspaces=true,
stepnumber=1,
numbers=none,
captionpos=b,
showspaces=false,
frame=none
}

\newcommand{\QQ}{\mathbb{Q}}

\newcommand{\CC}{\mathbb{C}}

\DeclareMathOperator{\init}{in}
\DeclareMathOperator{\tr}{Tr}

\title{Elimination Templates in Macaulay2}

\author[Batavia]{Manav Batavia}
\address{Department of Mathematics, Purdue University, 150 N University St., West Lafayette, IN~47907, USA}
\email{mbatavia@purdue.edu}

\author[Chen]{Cheng Chen}
\address{Department of Mathematics, University of Wisconsin, 480 Lincoln Dr., Madison, WI~53706, USA}
\email{chengchen@math.wisc.edu}

\author[Chlopecki]{Anna Natalie Chlopecki}
\address{Department of Mathematics, Purdue University, 150 N University St., West Lafayette, IN~47907, USA}
\email{achlopec@purdue.edu}

\author[Duff]{Timothy Duff}
\address{Department of Mathematics,  University of Missouri, 810 Rollins St, Columbia MO 65203, USA}
\email{tduff@missouri.edu}

\author[Huang]{William Huang}
\address{Department of Mathematics, University of Wisconsin, 480 Lincoln Dr., Madison, WI~53706, USA}
\email{whuang259@wisc.edu}

\author[Li]{Aolong Li}
\address{Meta Platforms, Inc., 1 Meta Way, Menlo Park, California 94025, USA}
\email{aolongli@meta.com}

\author[Shen]{Wanchun Shen}
\address{Department of Mathematics, Harvard University, 1 Oxford St., Cambridge, MA~02138, USA}
\email{wshen@math.harvard.edu}

\date{April 30, 2026}

\begin{document}

\maketitle

\textit{Abstract}: We introduce the package \texttt{EliminationTemplates} for the Macaulay2 computer algebra system, which provides tools for constructing automatic solvers for families of zero-dimensional radical ideals depending on algebraically independent parameters.
This article provides a self-contained description of how elimination templates are constructed for such families and their specialization properties.
Additionally, we describe the main functionality and datatypes provided by our package, and illustrate its usage on several examples, including applications from computer vision from which elimination templates originated.

\section{Introduction}\label{sec:introduction}

Finding the solutions to a system of polynomial equations is a common problem throughout mathematics and its various applications.
This problem may be very difficult in general.
However, it is quite common  in practice that these systems have a particular structure---typically, they depend on some number of input parameters, such that generic  parameter values produce polynomials generating a zero-dimensional ideal.

In Section~\ref{subsec:constructing-templates}, we recall the well-known fact that solving zero-dimensional polynomial systems can be reduced to solving a (non-symmetric) eigenvalue problem.
However, implementing this reduction in practice demands significant care, as it generally involves exact Gr\"{o}bner basis computation, whereas eigenvalue computation is inherently inexact.
Naively combining these two disparate computational paradigms may fail, for instance, due to (1) the discontinuity of Gr\"{o}bner bases to perturbations in polynomials' coefficients~\cite{DBLP:conf/aaecc/Mourrain99}, or by producing eigenvalue problems which may be (2) poorly-conditioned~\cite{DBLP:journals/jsc/MourrainTB21,DBLP:journals/ijcv/ByrodJA09,GRAF2026308} or (3) simply too large to be solved at all.
Failure (1), however, may be addressed by seeking reductions that are specific to the parametric family of interest; these may be stable to perturbations in the parameters of interest, rather than the coefficients they determine.
Failures of types (2)/(3) are generally unavoidable, but may be mitigated in practice by experimenting with different reduction strategies and identifying structural features of parametric families that enable solving more stably/efficiently than black-box approaches.

In this paper, we provide a self-contained introduction to the \emph{elimination template} approach to solving $0$-dimensional systems, and describe an implementation of this approach in the \texttt{EliminationTemplates} package for the computer algebra system \texttt{Macaulay2}~\cite{M2}.
Elimination templates are an incarnation of the Gr\"{o}bner trace method algorithm~\cite{traverso}, with extensive developments tailored towards applications in computer vision.
In these applications, elimination templates are used to construct \emph{minimal solvers} which repeatedly solve systems from a fixed parametric family for many different parameter values depending on data. The outputs of minimal solvers are later used to remove outlier data and initialize least-squares estimation.

To set the stage, we pose a simple minimal problem which can be solved using elimination templates.

\begin{example}\label{ex:intro-example}
Consider two lines in complex projective space, $L_1, L_2 \subset \mathbb{P}^3$, which are known to intersect. 
Our goal will be to recover a linear map $P:\mathbb{P}^3\to\mathbb{P}^2$ from $L_1$, $L_2$, and the images of these lines, $\ell_1 = P(L_1), \ell_2 = P(L_2) \subset \mathbb{P}^2.$
The map $P$ will model a camera with variable focal length rotating around the origin. Letting $f$ be the focal length of our camera lens, we may represent $P$ by a matrix of the form 
\begin{align*}
\begin{pmatrix}
f&0&0\\0&f&0\\0&0&1
\end{pmatrix}
\begin{pmatrix}
R\mid \textbf{0}
\end{pmatrix},
\end{align*}
where $R$ is a $3\times 3$ rotation matrix such that $RR^T=I_3$ and $\text{det}(R)=1$, i.e. $R\in \operatorname{SO}(3)$. Notice that the fourth column of the $3\times 4$ augmented matrix $(R\mid \textbf{0})$ is made up of $0$'s because the center of the camera is at the origin.\\
\indent Alternatively, $P$ may be described using a \emph{scaled quaternion parametrization}, which is the map
\begin{align}\label{eq:quaternion}
Q:\mathbb{P}^3& \to \mathbb{P} \left( \CC^{3\times 3} \right)    \\
[w:x:y:z] &\mapsto 
\begin{pmatrix}
w^2+x^2-y^2-z^2&2xy-2wz&2wy+2xz\\
2xy+2wz&w^2-x^2+y^2-z^2&-2wx+2yz\\
-2wy+2xz&2wx+2yz&w^2-x^2-y^2+z^2
\end{pmatrix} \nonumber 
\end{align}
Note that for a real-valued point in the domain $[w:x:y:z] \in \mathbb{P}_{\mathbb{R}}^3$, we may always assume homogeneous coordinates are normalized so that they satisfy 
\begin{equation}
    \label{eq:sphere-normalization} 
w^2+x^2+y^2+z^2=1.
\end{equation}
With this normalization, the matrix of~\eqref{eq:quaternion} lies in $\operatorname{SO}_{\mathbb{R}}(3)$. 
Conversely, for any real rotation $R\in \operatorname{SO}_{\mathbb{R}}(3)$, there exists a point in $[w:x:y:z] \in \mathbb{P}_{\mathbb{R}}^3$ with $Q([w:x:y:z]) = R.$ \\
\indent In Example~\ref{ex:quaternions}, we illustrate how to construct synthetic parameters $(L_1, L_2, \ell_1, \ell_2)$ and generate the elimination template of the corresponding zero-dimensional ideal, generated by $5$ equations (including~\eqref{eq:sphere-normalization}) in the $5$ unknowns $(f, w,x,y,z).$ 
We also show how the resulting template can be reused when presented with new parameters.
\end{example}

To date, software for generating elimination templates has been implemented in \texttt{Matlab}~\cite{Larsson,DBLP:journals/corr/abs-2004-11765} (using \texttt{Macaulay2} as a back-end) and \texttt{Maple}~\cite{greedy,DBLP:journals/corr/abs-2307-00320}.
Once an elimination template is determined, it is straightforward to derive a minimal solver in the implementer's language of choice.
Finally, we note the extensive literature~\cite{DBLP:conf/eccv/KukelovaBP08,DBLP:journals/ijcv/ByrodJA09,Larsson,greedy,DBLP:journals/corr/abs-2307-00320,DBLP:conf/eccv/LarssonA16,DBLP:conf/iccv/LarssonAO17} that might be used for future improvements of our package.

Section~\ref{sec:background} of this paper describes our approach to constructing elimination templates.
Our exposition largely follows the approach given in~\cite{greedy}, but is self-contained, relying only on  basic algebraic geometry at the level of~\cite{cox-using}.
Two novel aspects of our treatment are the construction of templates which are universal for generic linear forms and an emphasis on specialization properties for parametric systems.
In Section~\ref{sec:design}, we describe the main datatypes and functionality provided by our package.
Finally, in Section~\ref{sec:examples}, we illustrate the application of our package to Section~\ref{ex:intro-example} and related problems.

\section{Background on Elimination Templates}\label{sec:background}

\subsection{Constructing Elimination Templates}\label{subsec:constructing-templates}

Let $I\subset R := \KK[x_1,\ldots , x_n]$ be a zero-dimensional ideal over a field $\KK,$ so that the quotient ring $R/I$ is a finite-dimensional vector space over $\KK.$
When $I$ is radical, the vector space dimension $\dim_\KK (R/I)$ is equal to the number of points in the variety $V_{\overline{\KK}} (I),$ where $\overline{\KK} \supset \KK$ is an algebraic closure of $\KK.$
We will be largely interested in ideals defined over the field $\KK = \QQ$ of rational numbers, as well as points in the associated variety $V(I) $, whose coordinates belong to the fields $\mathbb{\overline{Q}} \subset \mathbb{C}.$ 
In practice, these coordinates are approximated in floating-point.
Two other important cases arise when $\KK$ is a function field generated by algebraically independent indeterminates (see Section~\ref{subsec:specializing}) and when $\KK$ is a finite field of prime characteristic (as illustrated in Example~\ref{ex:strategies}.)

Fix a monomial order $<$ on $R,$ so that $R/I$ has an ordered $\KK$-basis $\mathcal{B}$ of standard monomials, namely
\begin{equation}\label{eq:std-mon-basis}
\mathcal{B} = \{ x^{\alpha_1} , \ldots , x^{\alpha_d} \}, \quad x^{\alpha_i}  \notin \init_< (I) \quad   \forall i=1, \ldots , d, 
\quad 
\mathbf{0} = \alpha_d <  \cdots  < \alpha_2 < \alpha_1 .
\end{equation}
We denote the sorted vector of standard monomials by 
\begin{equation}\label{eq:vB}
v(\mathcal{B}) = \begin{pmatrix}
    x^{\alpha_1 } & \cdots & x^{\alpha_d} 
\end{pmatrix}^T
\in R^{d\times 1}.
\end{equation}
Multiplication by any element $f\in R$ determines a $\mathbb{K}$-linear transformation
\begin{align}
m_f : R/I &\to R/I\\
[g] &\mapsto [fg].
\end{align}
Fixing such an $f,$ we may write
\begin{align}
m_f ( [x^{\alpha_1}]) &= a_{11} [x^{\alpha_1}] + \ldots + a_{1d} [x^{\alpha_d}], \nonumber\\
\vdots \quad  \quad  & \quad \quad \quad \vdots \nonumber \\
m_f ( [x^{\alpha_d}]) &= a_{d1} [x^{\alpha_1}] + \ldots + a_{dd} [x^{\alpha_d}] \label{eq:act-f},
\end{align}
where $A= (a_{ij})\in \KK^{d\times d} $ is the \emph{action matrix} associated to the triple $(I, <,f).$ 

The following result is the basis of reducing $0$-dimensional solving to solving the eigenvalue problem.

\begin{theorem}\label{thm:eigenvector}(See e.g.~\cite[\S 2.4]{cox-using})
For any zero-dimensional radical ideal $I,$ the eigenvalues of the action matrix $A$ are precisely the values of $f$ on the points of $V (I).$
If $f$ takes $d= \dim_\KK (R/I)$ distinct values on these points, then every eigenvalue of $A$ has algebraic multiplicity $1,$ and every eigenvector of $A$ is a scalar multiple of $v(\mathcal{B})|_{x = P}$, the $d\times 1$ vector obtained by evaluating standard monomials at some point $P\in V (I).$
\end{theorem}
Gr\"{o}bner basis algorithms are typically needed to determine the quotient ring basis $\mathcal{B}$ and action matrix $A$.
The idea behind \emph{elimination templates} is this: for fixed $<$ and $f\in R$ (typically a general linear form), and two ideals $I,J\subset R$ of a ``similar structure" (see Section~\ref{subsec:specializing} for a precise statement), the full cost of computing a Gr\"{o}bner basis need only be incurred once for the ideal $I,$ after which the action matrix of $J$ may be determined using just linear algebra over the field $k$.

Let us assume that $f\in R$ is a general linear form so that, in particular, $f$ will have the point separation property required by Theorem~\ref{thm:CLOspecializationGrobner}.
Our construction of an elimination template adjoins a new variable $s$ to $R.$
By abuse of notation, we identify elements of $R$ with their images under the inclusion map $\iota : R\hookrightarrow R[s]$.
We equip $R[s]$ with the \emph{elimination order} $<_s$ defined by
\begin{equation}
s^a \mathbf{x}^{\mathbf{\alpha}} <_s s^b \mathbf{x}^{\mathbf{\beta}}
\quad 
\text{ if } \quad 
\begin{cases}
    a<b, &\quad \text{ or }\\
    a=b &\quad \text{ and }  \quad \mathbf{x}^{\mathbf{\alpha}} < \mathbf{x}^{\mathbf{\beta}}.
\end{cases}
\end{equation}
Finally, we define the \emph{graph ideal} associated to $I$ and $f,$
\begin{equation}\label{eq:extended-ideal}
I_{f\to s} :=I + \langle s-f\rangle \subset R[s].
\end{equation}
Passing from $I$ to the graph ideal $I_{f \to s}$ preserves the action matrix in a very precise sense.
\begin{proposition}\label{prop:elim-order}
The triples $(I, <, f)$ and $(I_{f \to s}, <_s , s)$ have the same action matrix. 
\end{proposition}
\begin{proof}
By construction, the basis of standard monomials~\eqref{eq:std-mon-basis} of $R/I$ is also a basis of standard monomials of  $R[s]/I_{f\to s}$, and the linear relations~\eqref{eq:act-f} lift to express the action of $m_s$ on this basis.
\end{proof}

Now, suppose the ideal $I= \langle f_1, \ldots , f_r \rangle $ is given in terms of a fixed generating set $F = \{ f_1, \ldots , f_r \}$, and that $I$ has a reduced Gr\"{o}bner basis $G = \{ g_1, \ldots , g_t \}$ with respect to $<.$
The vectors
\begin{equation}\label{eq:vF-vG}
v(F) := \begin{pmatrix}
    f_1 & \cdots 
&     f_r 
\end{pmatrix}^T \in R^{r\times 1},
\quad 
v(G) := \begin{pmatrix}
    g_1 & \cdots 
&     g_t 
\end{pmatrix}^T \in R^{t\times 1}
\end{equation}
may be related by a \emph{change matrix},
\begin{equation}
H_{F\to G} \in R^{t \times r} \quad \text{s.t.} \quad 
v(G) = H_{F\to G} \cdot v(F).
\end{equation}
From the basis of standard monomials~\eqref{eq:std-mon-basis} for $I$ with respect to $<,$ we may consider the following \emph{vector of action relations} in the ring $R$,
\begin{equation}\label{eq:V-vector}
v(A, R) := 
f \cdot v(\mathcal{B})
-
A
\cdot 
v(\mathcal{B}).
\end{equation}
By construction, each entry of $v(A, R)$ belongs to the ideal $I$.
Thus, using multivariate polynomial division, there exists another polynomial matrix
\begin{equation}
    H_{G \to A} \in R^{d\times t} \quad 
    \text{s.t.} \quad 
    v(A, R) = H_{G \to A } \cdot 
   v(G).
\end{equation}
Now, we set
\begin{equation}
H_0 = H_{G \to A} \cdot H_{F\to G }
\in R^{d \times r}.\label{EqH_{FtoA}}
\end{equation}
Working in the extended polynomial ring $R[s],$ we may write the action relations as follows: 
\begin{equation}\label{eq:H0matrix-Rs}
s \cdot v(\mathcal{B})
-
A
\cdot 
v(\mathcal{B})
=
(s-f) v(\mathcal{B}) 
+ v(A,R)
= (s-f) v(\mathcal{B}) + H_0 \cdot v(F) . 
\end{equation}
Equation~\eqref{eq:H0matrix-Rs} shows that the action of $A$ on standard monomials can be recovered by multivariate polynomial arithmetic, provided that the Gr\"{o}bner basis $G$ has already been computed. 
The construction of an elimination template recasts this multivariate polynomial computation as row-reduction of an associated Macaulay-type matrix whose entries lie in the ground field $\KK.$

\begin{remark}
    For an ideal $J$ ``similar'' to $I$, given the Macaulay-type matrix $H_0$ for $I$, the action matrix for $J$ can be computed without any Gr\"obner basis computation for $J$. See Section \ref{subsec:specializing} for details.
\end{remark}
\begin{remark}\label{remark:H}
Observe that equation~\eqref{eq:H0matrix-Rs} still holds if we replace $H_0$ with any matrix $H\in R^{d\times r}$ of the form $H = H_0 +\Theta \cdot H_1,$ where $\Theta \in R^{d\times p}$ and the rows of $H_1\in R^{p\times r}$ generate the first syzygy module of $F$:
\begin{equation}\label{eq:Hmatrix-Rs}
s \cdot v(\mathcal{B})
-
A
\cdot 
v(\mathcal{B})
=
 (s-f) v(\mathcal{B}) + H \cdot v(F). 
\end{equation}
This observation is the basis of strategies for reducing the size of an elimination template---see Section \ref{subsec:strategies}.
\end{remark}

For any matrix $H \in R^{d \times r}$ and each $i=1, \ldots , r$ we let $\mathcal{S}_i(H)$ denote the \emph{$i$-th shift set}, consisting of all monomials appearing in some entry of the $i$-th column of $H$. 
We denote the union of all shift monomials by
\begin{equation}
    \mathcal{S}(H) := \mathcal{S}_1 (H)\displaystyle\cup \ldots \displaystyle\cup \mathcal{S}_{r} (H).\label{EqShiftSet}
\end{equation}

Now, let $\mathcal{M} (H)$ denote the set of all monomials appearing in either $\mathcal{B}$, $s \cdot \mathcal{B},$ or in the support of some shifted generator of the form $x^{\beta_{ij}} f_i$ with $x^{\beta_{ij}} \in \mathcal{S}_i (H).$ 
Let us partition $\mathcal{M}( H)$ into three sets,
\begin{align}
\mathcal{B} (H) &:=  \mathcal{B}, \tag{the \emph{basic} monomials}\\
\mathcal{R} (H) &:= s \cdot \mathcal{B} , \tag{the \emph{reducible} monomials}\\
\mathcal{E} (H) &:= \mathcal{M} (H) \setminus \left( \mathcal{B} (H) \cup \mathcal{R} (H)\right) . \tag{the \emph{excessive} monomials}
\end{align}
We may form the following Macaulay-type matrix
\begin{equation}\label{eq:template-matrix-blocks}
\renewcommand{\arraystretch}{1.2}
\begin{blockarray}{cccc}
  & \mathcal{E}(H) & \mathcal{R}(H) & \mathcal{B}(H) \\
  \begin{block}{c(c|c|c)}
    x^{\beta_{11}} f_1 & \ast & 0 & \ast \\
    \vdots & \vdots & \vdots & \vdots \\
    x^{\beta_{r1}} f_r & \ast & 0 & \ast \\
    \vdots & \vdots & \vdots & \vdots \\
    (s-f) v(\mathcal{B}) & \boldsymbol{\ast} & \mathbf{I}_d & \boldsymbol{\ast} \\
  \end{block}
\end{blockarray}
\in \KK^{(d + \# \mathcal{S} (H)) \times (2d + \#\mathcal{E} (H))}.
\end{equation}
Now, using the polynomial identity~\eqref{eq:Hmatrix-Rs}, we may express the action relations as linear combinations of the rows of this matrix, which thus has a row-echelon form given block-wise by
\begin{equation}\label{eq:rref-block-template}
\left(
\begin{array}{c|c|c}
    \boldsymbol{\ast}  & \mathbf{0} & \phantom{-]}\boldsymbol{\ast}   \\
     \mathbf{0} & \mathbf{I}_d & -A  
\end{array}
\right).
\end{equation}
We see that, in order to recover the action matrix $A$ with row operations over $\KK,$ the columns corresponding to the reducible monomials $\mathcal{R} (H)$ may harmlessly be dropped.
\begin{definition}\label{def:e-template}
Let $I=\langle f_1, \ldots , f_s \rangle$ be a zero-dimensional ideal in the polynomial ring $R$, with $d\times d$ action matrix $A$ for some monomial order $<,$ and $H\in R^{s\times d}$ any matrix satisfying~\eqref{eq:Hmatrix-Rs}.
We define the $(\# \mathcal{S}(H) + d ) \times (\# \mathcal{E} (H) + d) $ \emph{template matrix} $M$ whose first $\# \mathcal{S} (H)$ rows are the coefficient vectors of each $x^{\beta_{ij}} f_i$ for each $x^{\beta_{ij}}\in \mathcal{S}_i$, and whose last $d$ rows are the the coefficient vectors of $(s-f) v(\mathcal{B}).$ 
The columns of $M$ are divided into consecutive blocks $\mathcal{E} (H)$ and $\mathcal{B} (H),$ each in decreasing order with respect to $<.$
Pictorially,
\begin{equation}\label{eq:template-matrix-blocks2}
M =
\renewcommand{\arraystretch}{1.2}
\begin{blockarray}{ccc} 
  & \mathcal{E}(H) & \mathcal{B}(H) \\
  \begin{block}{c(c|c)}
    x^{\beta_{11}} f_1 & \ast & \ast \\
    \vdots & \vdots & \vdots \\
    x^{\beta_{s1}} f_s & \ast & \ast \\
    \vdots & \vdots & \vdots \\
    (s-f) v(\mathcal{B}) & \boldsymbol{\ast} & \boldsymbol{\ast} \\
  \end{block}
\end{blockarray}
\in \KK^{(d + \# \mathcal{S} (H)) \times (d + \#\mathcal{E} (H))}.
\end{equation}

\end{definition}

Let us emphasize that both the action matrix $A$ and the template matrix $M$ associated to an ideal $I$ are determined by several choices: namely, the input generators $F,$ the monomial order $<,$ and the matrix $H$ in~\eqref{eq:Hmatrix-Rs}.
These choices may have significant influence on a solver in practice, as they effect both the runtime of row-reducing the template matrix $M$ and the conditioning of the eigenvalue problem for $A.$

\subsection{Specializing Elimination Templates}\label{subsec:specializing}

Suppose now that our ground field is the complex function field $\KK = \CC (p_1, \ldots , p_m)$. 
We think of $p=(p_1, \ldots , p_m)$ as parameters which specify family of ideals.
We would like to formalize the intuition that elimination templates for zero-dimensional ideals $I \subset \KK[x_1, \ldots , x_n]$ are preserved under specializations $\Psi_{P} : \KK[x_1,\ldots , x_n] \to \CC [x_1, \ldots , x_n]$ for generic $P \in \CC^m .$
To do this, suppose that $I$ is the extension of an ideal $\tilde{I} \subset S := \CC [x_1,\ldots , x_n, p_1, \ldots , p_m]$ under the inclusion map

\begin{equation} \label{eq:inclusion} 
\iota : S \hookrightarrow \KK[x_1, \ldots , x_n].\end{equation}

Let us furthermore assume that the parameters $p$ are independent in the sense that 

\begin{equation}\label{eq:param-independent}
\tilde{I}\cap \CC[p_1,\dots,p_m]=\{0\}.
\end{equation}

Under these conditions, Gr\"{o}bner bases for a particular class of monomial orders have well-known specialization properties.
We quote the following result from~\cite[Ch.~6, \S 3, Proposition 1]{clo-4e}.

\begin{theorem}\label{thm:CLOspecializationGrobner}
    Let $\tilde{I}$ be an ideal in $S$ such that~\eqref{eq:param-independent} holds. Fix an elimination order $<_p$ such that $x^{\alpha}<_px^{\beta}$ implies $x^{\alpha}p^{\gamma}<_px^{\beta}$ for all $\gamma$. Let $\tilde{G}=\{g_1,\dots,g_t\}$ be a reduced Gr\"obner basis of $\tilde{I}$.

    For $i=1,\dots,t$, write $g_i\in \tilde{G}$ in the form $$g_i=h_i(p_1,\dots,p_m)x^{\alpha_i}+\text{ (terms}<_p x^{\alpha_i}),$$ where $h_i\in \CC[p_1,\dots,p_m]$ is nonzero. Then, for all $P=(a_1,\dots,a_m)\in \CC^m\setminus V(h_1,\dots,h_t)$, the set $$G=\{g_1(x_1,\dots,x_n,a_1,\dots,a_m),\dots,g_t(x_1,\dots,x_n,a_1,\dots,a_m)\}$$ is a Gr\"obner basis of $\Psi_{P}(I)$ with respect to the order $<$ on $\CC[x_1,\dots,x_n]$ obtained by restricting $<_p$. 
\end{theorem}

Now, let $M$ be a template matrix for $I \subset \KK [x_1, \ldots , x_n],$ and $A$ the corresponding action matrix.
Let us assume a fixed row-echelon form for the template matrix $M$ of the form
\begin{equation}\label{eq:rref-template}
E_1 (p) \cdots E_\ell (p) \cdot M =
\left(
\begin{array}{c|c}
    \boldsymbol{\ast}  & \phantom{-}\boldsymbol{\ast}   \\
     \mathbf{0} &  -A  
\end{array}
\right)
,
\end{equation}
where $E_1 (p), \ldots , E_\ell (p) \in \operatorname{GL} (d + \# \mathcal{S} (H), \, \KK)$ are elementary matrices.
We make the following observations:
\begin{enumerate}
    \item For $1\le i \le \ell ,$ each of the conditions $\det (E_i (p)) \ne 0$ defines a nonempty Zariski-open subset of $\CC^m.$ Let $\mathcal{U}_1  \subset \CC^m$ denote the intersection of all such sets, so that specialization at generic $P\in \mathcal{U}_1$ preserves the row echelon form~\eqref{eq:rref-template}.
    \item Theorem~\ref{thm:CLOspecializationGrobner} defines a nonempty Zariski-open set $\mathcal{U}_2 \subset \CC^m$ for which Gr\"{o}bner bases specialize; in particular, this shows that the specialization of the action matrix for $I$ remains an action matrix for the specialized ideal $\Psi_P (I)$.
\end{enumerate}
Thus, setting $ \mathcal{U} = \mathcal{U}_1 \cap \mathcal{U}_2,$ we obtain the following conclusion.

\begin{theorem}\label{thm:specialization}
For $\KK = \mathbb{C} (p_1, \ldots , p_m)$ and $I \subset \KK[x_1, \ldots , x_n]$ any zero-dimensional ideal obtained as the extension of $\tilde{I}\subset S$ such that~\eqref{eq:param-independent} holds, 
and for any monomial order $<$ 
there exists a dense Zariski-open subset $\mathcal{U} \subset \mathbb{C}^m$ such that the elimination template for $I$ constructed in Definition~\ref{def:e-template} specializes to an elimination template for the specialized ideal $\Psi_P (I)$ for all $P \in \mathcal{U}.$ 
\end{theorem}

\begin{example}
To demonstrate Theorem~\ref{thm:specialization}, we consider the parametric ideal $$I=(x^2+ay^2-1,xy-b)\in\CC(a,b)[x,y].$$ 

The code below illustrates an echelon form~\eqref{eq:rref-template} involving parameters $a$ and $b.$

\begin{lstlisting}[language=macaulay2]
FF = frac(QQ[a,b,c,d]);
R = FF[x, y, MonomialOrder => Lex];
l = x - 2*y;
I = ideal(x^2 + a*y^2 - 1, x*y - b);
needsPackage "EliminationTemplates";
ET = eliminationTemplate(l, I);
M = getTemplateMatrix ET
(P, L, U) = LUdecomposition M;
\end{lstlisting}
We may inspect the template matrix and its $LU$ factors as follows:

\begin{lstlisting}[language=macaulay2]
i7 : M
o7 = {-4} | 1 0 0  0  0  0  a 0 0 0 0 0 -1 0  0  |
     {-3} | 0 1 0  0  0  0  0 0 0 0 0 a 0  -1 0  |
     {-4} | 1 0 0  0  -b 0  0 0 0 0 0 0 0  0  0  |
     {-3} | 0 1 0  0  0  -b 0 0 0 0 0 0 0  0  0  |
     {-4} | 0 0 1  0  0  0  0 0 0 0 0 0 -b 0  0  |
     {-3} | 0 0 0  1  0  0  0 0 0 0 0 0 0  -b 0  |
     {-2} | 0 0 0  0  1  0  0 0 0 0 0 0 0  0  -b |
     {-3} | 0 0 -1 0  0  0  2 1 0 0 0 0 0  0  0  |
     {-2} | 0 0 0  -1 0  0  0 0 1 0 0 2 0  0  0  |
     {-1} | 0 0 0  0  -1 0  0 0 0 1 0 0 2  0  0  |
     {0}  | 0 0 0  0  0  -1 0 0 0 0 1 0 0  2  0  |

              11       15
o7 : Matrix FF   <-- FF

i8 : L
o8 = | 1 0 0  0  0      0   0       0 0 0 0 |
     | 0 1 0  0  0      0   0       0 0 0 0 |
     | 0 0 1  0  0      0   0       0 0 0 0 |
     | 0 0 0  1  0      0   0       0 0 0 0 |
     | 1 0 0  0  1      0   0       0 0 0 0 |
     | 0 1 0  0  0      1   0       0 0 0 0 |
     | 0 0 0  0  (-1)/b 0   1       0 0 0 0 |
     | 0 0 -1 0  0      0   (-2b)/a 1 0 0 0 |
     | 0 0 0  -1 0      0   0       0 1 0 0 |
     | 0 0 0  0  1/b    0   -1      0 0 1 0 |
     | 0 0 0  0  0      1/b 0       0 0 0 1 |

              11       11
o8 : Matrix FF   <-- FF

i9 : U
o9 =  | 1 0 0 0 0  0  a      0 0 0 0 0   -1        0        0        |
      | 0 1 0 0 0  0  0      0 0 0 0 a   0         -1       0        |
      | 0 0 1 0 0  0  0      0 0 0 0 0   -b        0        0        |
      | 0 0 0 1 0  0  0      0 0 0 0 0   0         -b       0        |
      | 0 0 0 0 -b 0  -a     0 0 0 0 0   1         0        0        |
      | 0 0 0 0 0  -b 0      0 0 0 0 -a  0         1        0        |
      | 0 0 0 0 0  0  (-a)/b 0 0 0 0 0   1/b       0        -b       |
      | 0 0 0 0 0  0  0      1 0 0 0 0   (-ab+2)/a 0        (-2b2)/a |
      | 0 0 0 0 0  0  0      0 1 0 0 2   0         -b       0        |
      | 0 0 0 0 0  0  0      0 0 1 0 0   2         0        -b       |
      | 0 0 0 0 0  0  0      0 0 0 1 a/b 0         (2b-1)/b 0        |

               11       15
o9 : Matrix FF   <-- FF
\end{lstlisting}

Inspecting the entries of $U$, we should enforce the condition $ab\ne 0$ in order for the elimination template to specialize. Furthermore, computing a Gr\"{o}bner basis with respect to the given lexicographic order shows that the same conditions are sufficient for specialization, giving us the open set
\[ \mathcal{U} = \{(a, b) \in \CC^2 \mid  ab \neq 0\} \]
promised by Theorem~\ref{thm:specialization}.
It may also be of interest to consider a smaller open subset of specializations $P=(a,b)\in \mathcal{U}$ for which the parametric action matrix $A$ has four distinct eigenvalues. We find $A$ below:
\begin{lstlisting}[language=Macaulay2]
i10 : getActionMatrix ET
o10 = {-3} | 0    (ab-2)/a 0         2b2/a |
      {-2} | -2   0        b         0     |
      {-1} | 0    -2       0         b     |
      {0}  | -a/b 0        (-2b+1)/b 0     |

               4       4
o10 : Matrix FF  <-- FF
\end{lstlisting}
By computing the discriminant of its characteristic polynomial, we see that this parametric action matrix has four distinct eigenvalues for all parameters in the open set
\[
\{ (a,b) \in \mathcal{U} \mid a-4 \ne 0, \, 4 ab^2 -1 \ne 0, \, ab + 4b - 2 \ne 0\} .
\]
\end{example}

\subsection{Using Elimination Templates to Solve $0$-Dimensional Systems}\label{subsec:zero-dim}

For completeness, we briefly describe how Theorem~\ref{thm:eigenvector} can be used to find the points in $V(I).$ We retain the notation established in Theorem~\ref{subsec:constructing-templates}. Assume that $f\in R$ is chosen such that all eigenvalues of the action matrix have algebraic multiplicity $1$.
If $\lambda$ is one such eigenvalue, then Theorem~\ref{thm:eigenvector} gives a corresponding eigenvector $$v = v( \mathcal{B})_{x=a} = (a^{\alpha_1}, \ldots , a^{\alpha_{d-1}}, 1)$$ 
for some $a=(a_1,\ldots , a_n)\in V(I).$
When $f$ is a general linear form, we may assume its support contains every variable in the ring $R.$ 
Thus, using the template construction of Definition~\ref{def:e-template}, we see that every variable appears either in the set of basic variables $\mathcal{B} (H)$ or in the set of excessive monomials $\mathcal{E} (H).$ 
In the former case, $x_i \in \mathcal{B} (H),$ and the coordinate $a_i$ can be located as a coordinate of the eigenvector $v.$

In the latter case, taking $H= H_0$, any variable $x_i\in \mathcal{E} (H_0)$  must still appear as a pivot column in $M$.
In other words, the template matrix encodes an equation linear in the variable $x_i.$
Let us list all such variables in order, $x_{i_1},\dots,x_{i_l}$ with $i_1<\cdots<i_l.$ For $j=1,\dots,l$, there are elements $g_j\in G$ such that $$g_j=x_{i_j}+ \text{ terms involving } x_i \text{ for } i>i_j.$$ On evaluating at $P=(a_1,\dots,a_n),$ we obtain $$0=a_{i_j}+\text{ terms involving } a_i \text{ for } i>i_j.$$ We can first solve for $a_{i_l}$ as the values of $a_i$ for $i>i_l$ are known to us from previous computations. Once we know $a_{i_l}$, we can solve for $a_{i_{l-1}}$, and so on.

\section{Package Design and Functionality}\label{sec:design}

In this section, we will describe the functionality provided by the \texttt{EliminationTemplates} package and various strategies which implement the constructions of the previous section.

Thus, in general, all coordinates of all points in $V(I)$ may be recovered by row-reducing $M$ and computing the eigenvalues of $A.$ 
For numerical stability, we row-reduce $M$ using $LU$ decomposition with partial pivoting.

\subsection{Data Structures and Functions}\label{subsec:data-functions}

Our package consists of one main data type, \texttt{EliminationTemplate}. One way to call the corresponding \texttt{eliminationTemplate} constructor is to pass arguments of class \texttt{RingElement}, representing the action variable $f$ defining the action matrix, and a zero-dimensional ideal $I$, of class \texttt{Ideal}. Objects in this class store the necessary data required for an elimination template computation. Objects of class \texttt{EliminationTemplate} are designed to cache information such as the monomial basis $\mathcal{B}$, the template matrix $M$, and the action matrix $A$, so that they only need to be computed once. 

We illustrate our package data structures and functionality with a running example. First, we create a new template object, as shown in the code below:

\begin{center}
\begin{lstlisting}[language=macaulay2]
i1: needsPackage "EliminationTemplates";
i2: R = QQ[x,y];
i3: I = ideal(x^4+x*y+y^2-3, x^2*y+y^3-2);
i4: E = eliminationTemplate(x,I)
o4: action variable: x
\end{lstlisting}
\end{center}

The method \texttt{getTemplate} accepts an \texttt{EliminationTemplate} object and outputs a tuple of objects of type \texttt{List} containing a \texttt{ShiftSet}, as in (\ref{EqShiftSet}), and \texttt{MonomialPartition} associated to the matrix $H$ in (\ref{EqH_{FtoA}}). Alternatively, we can call \texttt{getTemplate} using the inputs \texttt{RingElement}, \texttt{Matrix}, which contains a basis for $R/I$, and \texttt{Ideal} $I$. The user is not expected to interact with the internal data types \texttt{ShiftSet} and \texttt{MonomialPartition}, which are cached inside of the \texttt{EliminationTemplate} object.
The \texttt{getTemplateMatrix} method takes in an object of class \texttt{EliminationTemplate} and returns the associated template matrix.

\begin{center}
\begin{lstlisting}[language=macaulay2]
i5: getTemplateMatrix(E)
o5: | 1 0 0 0 0 0 0  0 1 1 0 0 0 0  0 -3 0 0  0  |
    | 0 1 0 0 0 0 0  0 0 0 0 1 1 0  0 0  0 -3 0  |
    | 0 0 0 0 1 0 0  0 0 0 0 0 0 0  1 1  0 0  -3 |
    | 1 0 1 0 0 0 -2 0 0 0 0 0 0 0  0 0  0 0  0  |
    | 0 1 0 1 0 0 0  0 0 0 0 0 0 -2 0 0  0 0  0  |
    | 0 0 0 0 0 1 0  0 0 1 0 0 0 0  0 0  0 -2 0  |
    | 0 0 0 0 0 0 1  0 0 0 0 0 1 0  0 0  0 0  -2 |
\end{lstlisting}
\end{center}

The method \texttt{templateSolve} takes as input either an object of class \texttt{EliminationTemplate} or a \texttt{RingElement} $f$ and an \texttt{Ideal} $I$ with common polynomial ring $R$, and outputs a list of solutions. 
Each solution is a list of complex coordinates, one per ring variable. We exemplify the use of this function below.

\begin{center}
\begin{lstlisting}[language=macaulay2]
i6: sols1 = templateSolve(E);
i7: sols2 = templateSolve(x,I);
i8: assert(all(sols1, x -> 1e-6 > norm sub(sub(gens I, QQ[gens R]), matrix{x})))
o8: true
\end{lstlisting}
\end{center}

Using \texttt{copyTemplate}, we can create a copy of an \texttt{EliminationTemplate}, resulting in an \texttt{EliminationTemplate} instance defined using the same action variable but a different defining ideal.

\begin{center}
\begin{lstlisting}[language=macaulay2]
i9: J = ideal(x^4 + 2*x*y + y^2-1, 3*x^2*y + y^3 - 5);
i10: F = copyTemplate(E, J);
i11: sols3 = templateSolve(F);
i12: assert(all(sols3, x -> 1e-6 > norm sub(sub(gens J, QQ[gens R]), matrix{x})))
o12: true
\end{lstlisting}
\end{center}

\subsection{Additional Strategies}\label{subsec:strategies} 
There are three strategies which we allow the user to manually input, each which determine the matrix $H$ in equation~\ref{EqH_{FtoA}}, namely 
\begin{enumerate}[label=\arabic*.]
\item the default strategy
\item the strategy \texttt{"Larsson"}, implementing the syzygy-based reduction of $H$~\cite{Larsson}
\item the strategy \texttt{"Greedy"}, implementing the linear algebra-based reduction of $H$ introduced in~\cite{greedy}.
\end{enumerate}

As in Remark~\ref{remark:H}, equation~\eqref{eq:Hmatrix-Rs} remains valid if $H$ is replaced by any matrix of the form
\[
H = H_0 + \Theta \cdot H_1,
\]
where the rows of $H_1$ generate the first syzygy module of $F$.

In the default strategy, we simply set $H = H_0$ as in Equation~\eqref{EqH_{FtoA}}. No reduction is performed, and the resulting template is typically far from optimal.

In the Larsson strategy, we exploit the fact that $H_0$ can be replaced by $H_0 + \Theta \cdot H_1$, where the rows of $H_1$ generate the first syzygy module of $F$. We compute a Gröbner basis of the syzygy module with respect to a degree-first term order and reduce each row of $H_0$ to its normal form. This minimizes the maximum degree among equivalent representations~\cite{Larsson}, which often reduces template size in practice. 

In the greedy strategy, the matrix $\Theta$ is selected greedily, with the goal of reducing the number of polynomial shifts and excessive monomials, and hence the size of the template at each step. This is more expensive than the Larsson strategy, but directly targets template size. We refer to~\cite{greedy} for more details.
Example~\ref{ex:strategies} below illustrates how different strategies may be used to reduce template size.

\section{Examples}\label{sec:examples}

\begin{example}\label{ex:specialization}

As a basic demonstration, suppose we wanted to find all complex solutions to the system
\begin{equation}
    \begin{cases}
        x^2+y^2-1 &= 0 \\
        x^2+y^3+xy-2 &= 0
    \end{cases}
\end{equation}
To do this, we use the linear form $f = x + 4y$ to construct our elimination template.
\begin{center}
\begin{lstlisting}[language=macaulay2]
i1: R = QQ[x,y];
i2: I = ideal(x^2+y^2-1,x^2+y^3+x*y-2);
i3: E = eliminationTemplate(x+4*y, I);
\end{lstlisting}
\end{center}
The elimination template yields a set of solutions, which we can then check
\begin{center}
\begin{lstlisting}[language=macaulay2]
i4: sols = templateSolve(E);
i5: all(sols, x -> 1e-6 > norm sub(sub(gens I, QQ[gens R]), matrix{x}))
o5: true
\end{lstlisting}
\end{center}
We now illustrate how \texttt{copyTemplate} can be used to exploit the specialization property described in Theorem~\ref{thm:specialization}, by using the previously-constructed template to solve the new system given as follows:
\begin{equation}
    \begin{cases}
        x^2+y^2-2 &= 0 \\
        x^2+y^3+3xy-5 &= 0
    \end{cases}
\end{equation}
To solve this new system, we simply run \texttt{copyTemplate} and then \texttt{templateSolve}.
\begin{center}
\begin{lstlisting}[language=macaulay2]
i6: J = ideal(x^2+y^2-2,x^2+y^3+3*x*y-5);
i7: F = copyTemplate(E, J);
i8: sols = templateSolve(F);
i9: all(sols, x -> 1e-6 > norm sub(sub(gens J, QQ[gens R]), matrix{x}))
o9: true
\end{lstlisting}
\end{center}

\end{example}

\begin{example}\label{ex:quaternions}
We now explain in detail how our package may be used to solve the camera pose problem introduced in Example~\ref{ex:intro-example}. First we generate synthetic data representing two general lines $L_1, L_2$ in $\mathbb{P}^3$:
\begin{center}
\begin{lstlisting}[language=macaulay2]
i1: FF = QQ;
i2: a = random(FF^3, FF^1) || matrix{{1}};
i3: b1 = random(FF^3, FF^1) || matrix{{1}};
i4: b2 = random(FF^3, FF^1) || matrix{{1}};
\end{lstlisting}
\end{center}
We implement a function \texttt{Q2R} which implements the quaternion map~\eqref{eq:quaternion}
\begin{center}
\begin{lstlisting}[language=macaulay2]
i5: Q2R = (w,x,y,z) -> matrix{
        {w^2+x^2-y^2-z^2, 2*x*y-2*w*z, 2*w*y+2*x*z},
        {2*x*y+2*w*z, w^2-x^2+y^2-z^2, -2*w*x+2*y*z},
        {-2*w*y+2*x*z, 2*w*x+2*y*z, w^2-x^2-y^2+z^2}
    };
\end{lstlisting}
\end{center}
Next, we need to give two lines $\ell_1, \ell_2\subset \mathbb{P}^2$. In practice these data would come from experiments using a physical camera which would ensure the existence of a solution. For the purpose of generating a template, however, we may generate a ground-truth camera and use it to generate the line images:
\begin{center}
\begin{lstlisting}[language=macaulay2]
i6: (w0, x0, y0, z0, f0) := (random FF, random FF, random FF, random FF, random FF);
i7: R0 := Q2R(w0, x0, y0, z0);
i8: P0 := diagonalMatrix{f0, f0, 1} * (R0 | matrix{{0},{0},{0}});
i9: l1 = gens ker transpose(P0 * (a | b1));
i10: l2 = gens ker transpose(P0 * (a | b2));
\end{lstlisting}
\end{center}
Here $w_0,x_0,y_0,z_0,f_0$ are the camera parameters we wish to recover which in practice would come from a physical camera. In the above code, the lines $\ell_i$ are represented using implicit equations. Using this, we generate the camera matrix $P$ and the ideal which represents the constraints on $P$.
\begin{center}
\begin{lstlisting}[language=macaulay2]
i11: S = FF[w..z, f];
i12: R = Q2R(w,x,y,z);
i13: P = diagonalMatrix{f,f,1} * (R | matrix{{0},{0},{0}});
i14: I = ideal(w^2+x^2+y^2+z^2-1, transpose l1 * P * a, transpose l1 * P * b1, transpose l2 * P * a, transpose l2 * P * b2);
\end{lstlisting}
\end{center}
Now we can use our package to solve for the camera parameters $w_0,x_0,y_0,z_0,f_0$.
\begin{center}
\begin{lstlisting}[language=macaulay2]
i15: l = random(1, ring I);
i16: ET = eliminationTemplate(l, I);
i17: sols = templateSolve ET;
\end{lstlisting}
\end{center}
Finally, we verify that the ground truth camera is recovered as one of the solutions.
\begin{center}
\begin{lstlisting}[language=macaulay2]
i18: groundTruthSolution = matrix{(1/sqrt(w0^2+x0^2+y0^2+z0^2)*{w0,x0,y0,z0})|{f0}};
i19: position(sols, x -> norm(matrix{x} - groundTruthSolution) < 1e-10)
o19: {0}
\end{lstlisting}
\end{center}
In this formulation, the elimination template solver is efficient with the default settings, and the action matrix is $16 \times 16.$ Inspecting the solutions, we see that there are a number of symmetries; these arise (for example) due to the $2$-$1$ nature of the quaternion map~\eqref{eq:quaternion}.
Thus, symbolic symmetry-reduction techniques~\cite{DBLP:conf/eccv/LarssonA16} and analysis of this minimal problem's Galois group~\cite{duff2025galois,duff2022galois} would likely lead to further improvements in efficiency. We leave the task of integrating these tools into the \texttt{EliminationTemplates} package as future work.
\end{example}

\begin{example}\label{ex:computer-vision}
The problem of 5-point essential matrix estimation is an important and well-known minimal problem, which is widely used for estimating the relative orientation of two calibrated cameras. The first practical solution to this problem, due to Nist\'{e}r~\cite{nister2004efficient}, involved a careful derivation reducing this problem to computing the roots of a degree-$10$ polynomial.
The code below provides one means of automating this reduction.
Our goal is to estimate three variables $x,y,z$ representing the  matrix
\begin{equation}
    E = x E_1 + y E_2 + z E_3 + E_4,
\end{equation}
where we are given the $3 \times 3$ $E_1, E_2, E_3, E_4$. This matrix subject to 10 cubic Demazure constraints:
\begin{equation}
    EE^TE - \frac{1}{2} \tr (E^TE)E = 0 ,
    \quad 
    \det E = 0.
\end{equation}
Given rational-valued data $E_1, \ldots , E_3 \in \mathbb{Q}^{3\times 3},$ these constraints give us 10 polynomials in $\QQ[x,y,z]$. As in the previous example, we generate our template using random rational data.

\begin{center}
\begin{lstlisting}[language=macaulay2]
i1: setRandomSeed 0;
i2: R = QQ[x, y, z];
i3: Es = apply(4, i -> random(QQ^3, QQ^3));
i4: E = x * Es#0 + y * Es#1 + z * Es#2 + Es#3;
i5: I = ideal(E * transpose E * E - (1/2) * trace(E * transpose E) * E, det E);
i6: ET = eliminationTemplate(l, I);
i7: sols = templateSolve(ET);
i8: all(sols, x -> 1e-6 > norm sub(sub(gens I, CC[gens R]), matrix{x}))
i10: true
\end{lstlisting}
\end{center}
The optimal template size for this example is simply the $10\times 20$ matrix with rows indexed by the Demazure constraints; in other words, the rows depending on $s$ in Definition~\ref{def:e-template} are omitted.
\begin{center}
\begin{lstlisting}[language=macaulay2]
i11 :  M = getTemplateMatrix ET;

               10       20
o11 : Matrix QQ   <-- QQ        
    \end{lstlisting}
\end{center}
Writing $M = \left( \begin{array}{c|c} M_1 & M_2 \end{array} \right)$ with $M_1, M_2 \in \mathbb{Q}^{10\times 10},$ and assuming the parameters are sufficiently generic, we see that the associated action matrix is simply $M_1^{-1} M_2$. 
In practice, the parameters $E_0, \ldots , E_3$ determining this matrix are obtained by from matched pairs of 2D image points.
When solving for new parameter values, we may use \texttt{copyTemplate} just as we did in Example~\ref{ex:specialization}.
\begin{center}
\begin{lstlisting}[language=macaulay2]
i12: Es = apply(4, i -> random(QQ^3, QQ^3));
i13: E = x * Es#0 + y * Es#1 + z * Es#2 + Es#3;
i14: J = ideal(E*transpose E * E - (1/2) * trace(E * transpose E) * E);
i15: ET' = copyTemplate(ET, J);
i16: sols = templateSolve(ET');
i17: all(sols, x -> 1e-6 > norm sub(sub(gens J, CC[gens R]), matrix{x}))
o17: true
\end{lstlisting}
\end{center}

\end{example}

\begin{example}\label{ex:strategies}
Our last example illustrates the flexibility of the \texttt{EliminationTemplates} package for solving larger minimal problems and the different strategies that may be used. 
Like the previous example, this problem comes from estimating the relative pose of two cameras.
In this case, however, the two cameras share an unknown focal length. 
This leads to the following generalized Demazure constraints 
\begin{equation}
\det F = 0,
\end{equation}
\begin{equation}
F Q F^T Q F - \tfrac{1}{2} \tr(F Q F^T Q) F = 0,
\end{equation}
for the $3\times 3$ polynomial matrices
\begin{equation}\label{eq:F-Q-constraints}
F = F_1 + y F_2 + z  F_3,
\quad 
Q = \operatorname{diag} (1, 1, x),
\end{equation}
where $F_1, F_2 , F_3$ are $3\times 3$ matrices of parameters and and $x,y,z$ are unknown as in Example~\ref{ex:computer-vision}. 
For more background on this problem, we refer the reader to sources~\cite{MR3833649} and~\cite{greedy} which study similar formulations.

For this larger minimal problem, it is more effective to solve the \emph{offline problem} of constructing the template matrix over a finite field $\mathbb{K} = \mathbb{Z} / p \mathbb{Z}$ of large, prime characteristic.
This is a standard trick which allows the offline phase to run faster by reducing coefficient swell in Gr\"{o}bner basis calculations; using rational reconstruction techniques, it is possible to recover a template matrix over $\mathbb{Q}$ from several carefully chosen primes. 
Below, we simply fix $p=32749$ for illustrative purposes.

\begin{lstlisting}[language=macaulay2]
i1: setRandomSeed 0
i2: FF = ZZ/32749;
i3: R = FF[x,y,z];
i4: Fs = apply(4, i -> random(QQ^3, QQ^3));
i5: F = Fs#0 + y*Fs#1 + z*Fs#2;
i6: Q   = diagonalMatrix{1_R, 1_R, x};
i7: I = ideal(2*F*Q*transpose(F)*Q*F -trace(F*Q*transpose(F)*Q)*F) + ideal(det F);
\end{lstlisting}

For this example, we compare the template sizes produced by the default strategy and optional strategies \texttt{"Larsson"} and \texttt{"Greedy"} below.
As the monomial order in our code above is unspecified, all three templates use the \texttt{GRevLex} order with $x>y>z.$

As determined by the template size alone, the \texttt{"Greedy"} strategy is the top-performer, followed by \texttt{"Larsson"}, then finally the default strategy. This is shown below.

\begin{lstlisting}[language=macaulay2]
i8: ET = eliminationTemplate(x, I);
i9: getTemplateMatrix(ET);
               53       73
o9 : Matrix FF   <-- FF
i10:
getTemplateMatrix(ET, Strategy => "Larsson");   -- Larsson  :  53 x  73
               53       73
o10 : Matrix FF   <-- FF
i11:
getTemplateMatrix(ET, Strategy => "Greedy")    -- Greedy   :  31 x  50

               31       50
o11 : Matrix FF   <-- FF
\end{lstlisting}

The user running these lines of code will surely observe that, in terms of the time needed to generate the template, the order of these methods is reversed; the default strategy is faster than \texttt{"Larsson"}, and both are significantly faster than \texttt{"Greedy"}.
However, template generation is an offline task which, in principle, only needs to be performed once. 
To properly evaluate the \emph{online} minimal solvers produced by these templates, one must copy each template into a ring of characteristic $0$ and measure both the runtime and residuals produced by \texttt{templateSolve}, as we have demonstrated in previous examples.
The \texttt{EliminationTemplates} package enables users, for the first time, to perform the complete workflow of constructing, copying, and evaluating the performance of template-based solvers entirely within \texttt{Macaulay2}.

\end{example}

\section*{Acknowledgements}

The Macaulay2 workshop at University of Wisconsin-Madison in June-July 2025 served as a first introduction for many of the authors. We thank the organizers, the Department of Mathematics at University of Wisconsin-Madison, and the National Science Foundation for supporting the workshop through NSF Grant DMS-2508868.
We also thank Ikenna Nometa for helpful conversations early on.
MB was supported by NSF Grants DMS-2302430 and DMS-2100288, and by Simons Foundation Grant SFI-MPS-TSM-00012928.

\bibliographystyle{amsplain}
\bibliography{refs.bib}
\end{document}